\documentclass[12pt, twoside]{article}
\usepackage{graphicx,color,enumerate,hyperref,epsf,tabularx,booktabs}
\usepackage{amsfonts,amsthm,amsmath,amssymb,bm,times}

\def\titlerunning#1{\gdef\titrun{#1}}
\makeatletter
\def\author#1{\gdef\autrun{\def\and{\unskip, }#1}\gdef\@author{#1}}
\def\address#1{{\def\and{\\\hspace*{18pt}}\renewcommand{\thefootnote}{}%
\footnote {#1}}%
\markboth{\autrun}{\titrun}}
\makeatother
\def\email#1{e-mail: #1}
\def\subjclass#1{{\renewcommand{\thefootnote}{}%
\footnote{\emph{Mathematics Subject Classification (2010):} #1}}}
\def\keywords#1{\par\medskip
\noindent\textbf{Keywords.} #1}

\def\N{\mathbb{N}}

\newtheorem{theorem}{Theorem}[section]
\newtheorem{definition}[theorem]{Definition}
\newtheorem{lemma}[theorem]{Lemma}
\newtheorem{corollary}[theorem]{Corollary}
\newtheorem{proposition}[theorem]{Proposition}

\numberwithin{equation}{section}

\begin{document}

\baselineskip=17pt

\titlerunning{Ordering starlike trees}

\title{Ordering starlike trees by the totality of their spectral moments}

\author{Dragan Stevanovi\'c}

\date{}

\maketitle

\address{D. Stevanovi\'c: Mathematical Institute, Serbian Academy of Sciences and Arts, 
                          Kneza Mihaila 36, 11000 Belgrade, Serbia; \email{dragance106@yahoo.com}}


\begin{abstract}
The $k$-th spectral moment $M_k(G)$ of the adjacency matrix of a graph~$G$ 
represents the number of closed walks of length~$k$ in~$G$.
We study here the partial order $\preceq$ of graphs,
defined by $G\preceq H$ if $M_k(G)\leq M_k(H)$ for all $k\geq 0$,
and are interested in the question 
when is $\preceq$ a linear order within a specified set of graphs?
Our main result is that $\preceq$ is a linear order on 
each set of starlike trees with constant number of vertices.
Recall that a connected graph $G$ is a starlike tree
if it has a vertex~$u$ such that the components of $G-u$ are paths,
called the branches of~$G$.
It turns out that the $\preceq$ ordering of starlike trees with constant number of vertices 
coincides with the shortlex order of sorted sequence of their branch lengths.

\end{abstract}

\section{Introduction}
\label{sc-intro}

Let $G=(V,E)$ be a simple, connected graph with at least one edge.
A walk of length~$k$ in~$G$ is a sequence of its vertices $W\colon v_0,\dots,v_k$
such that $v_iv_{i+1}$ is an edge of~$G$ for each $i=0,\dots,k-1$.
It is a folklore result in graph theory that 
the number of walks of length~$k$ between vertices $u$ and~$v$ is equal to~$A^k(G)_{u,v}$,
where $A(G)$ is the adjacency matrix of~$G$.

The number of all walks $W_k(G)$ of length $k$ in $G$ is thus the sum of entries of $A^k(G)$.
Inequalities between numbers of walks of different lengths and other graph parameters
attracted attention of many researchers,
and we mention in passing a few such results.
Let $d_v$ denote the degree of a vertex $v$, and 
let $\bar d=2|E|/n$ denote the average vertex degree,
where $n=|V|$ is the number of vertices of~$G$.
Erd\"os and Simonovits~\cite{ersi82} proved that $n\bar{d}^k\leq W_k(G)$ for $k\in\N$,
while Fiol and Garriga~\cite{figa09} proved that $W_k(G)\leq\sum_{v\in V} d_v^k$,
which is a special case of an older result by Hoffman~\cite{ho67}.
Dress and Gutman~\cite{drgu03} showed that $W^2_{a+b}(G)\leq W_{2a}(G)W_{2b}(G)$ for $a,b\in\N$,
while T\"aubig~\cite{ta17} showed that $\left(W_a(G)/n\right)^b\leq W_{ab}(G)/n$.
Further examples of inequalities of this type are surveyed in a recent book by T\"aubig~\cite{ta17}.

A walk $W\colon v_0,\dots,v_k$ in~$G$ is closed if $v_0=v_k$.
Let $M_k(G,v)$ denote the number of closed walks of length~$k$ starting and ending at a vertex~$v$ of~$G$.
The sequence of numbers $M_k(G,v)$, $k\geq 2$,
provides a certain glimpse into the density of edges in the vicinity of~$v$.
For example, $M_2(G,v)$ is equal to the degree of~$v$, 
$M_3(G,v)$ is equal to twice the number of triangles containing~$v$,
while for larger values of~$k$, 
$M_k(G,v)$ counts a mix of closed walks going up to the distance $\lfloor k/2\rfloor$ from~$v$.
Cumulatively, 
let $M_k(G) = \sum_{v\in V} M_k(G,v)$ denote the total number of closed walks of length~$k$ in~$G$.

The numbers of all walks and closed walks are closely related to spectral properties of a connected graph.
Let $\lambda_1\geq\dots\geq\lambda_n$ be the eigenvalues of the adjacency matrix $A(G)$,
with $x_1,\dots,x_n$ the appropriate eigenvectors forming an orthonormal basis.
From the spectral decomposition~\cite{st15}
$$
A(G)=\sum_{i=1}^n \lambda_i x_i x_i^T
$$
and orthonormality of eigenvectors we have $A(G)^k=\sum_{i=1}^n \lambda_i^k x_i x_i^T$,
so that
\begin{equation}
\label{eq-spectral-moment}
W_k(G) = \sum_{i=1}^n \lambda_i^k \left(\sum_{u\in V} x_{i,u}\right)^2,
\quad\mbox{and}\quad
M_k(G) = \sum_{i=1}^n \lambda_i^k,
\end{equation}
i.e., the numbers of closed walks at the same time represent the spectral moments of~$A(G)$.
From the Perron-Frobenius theorem~\cite[Chapter XIII]{ga59}
which implies that $\lambda_1\geq|\lambda_i|$ for each $i=2,\dots,n$ and
that the entries of~$x_1$ are positive when $G$ is connected and has at least one edge,
we further get~\cite{cv71}
\begin{equation}
\label{eq-lambda-limit}
\lambda_1 = \lim_{k\to\infty} \sqrt[k]{W_k(G)} = \lim_{k\to\infty} \sqrt[2k]{M_{2k}(G)}. 
\end{equation}
Closed walks of even length are taken above to avoid dealing separately with bipartite graphs,
which do not have closed walks of odd lengths.

Lexicographical ordering of graphs by spectral moments
had been used earlier in producing graph catalogues~\cite{cvpe84}.
Some theoretical properties of such orderings, 
mostly within the sets of trees and unicyclic graphs,
have been reported in the literature~\cite{cvro87,wuli10,pahulili11,chli12,palili12,chlili12,lihu16}.
However, while helping to produce ordering of graphs 
from the sparsest to the densest in an intuitive sense,
lexicographical ordering by spectral moments does not have implications 
on behaviour of spectral radius or other spectral properties of graphs.
We are therefore rather interested in the following partial order of graphs
that takes into account the totality of their spectral moments.

\begin{definition}
For two graphs $G$ and $H$,
let $G\preceq H$ if $M_k(G)\leq M_k(H)$ for each $k\geq 0$.
Further, let $G\prec H$ if $G\preceq H$ and there exists $k'\geq 0$ such that $M_{k'}(G)<M_{k'}(H)$.
\end{definition}

Hence from (\ref{eq-lambda-limit}) we have
$$
G\preceq H \quad\Rightarrow\quad \lambda_1(G)\leq\lambda_1(H).
$$
The $\preceq$ order also has implications on the Estrada index of graphs.
Estrada and Highman \cite[Section 3]{eshi10} proposed the use of a weighted series of the numbers of closed walks
$$
f_c(G,v) = \sum_{k\geq 0} c_k M_k(G,v)
$$
as a descriptor of complex networks,
where $(c_k)_{k\geq 0}$ is a predefined sequence of nonnegative weights that makes the series convergent.
Values $f_c(G,v)$ may then be considered as the closed walk based measure of vertex centrality,
while
$$
f_c(G) = \sum_{v\in V} f_c(G,v) = \sum_{k\geq 0} c_k M_k(G)
$$
represents a cumulative closed walk based descriptor of a network.
The Estrada's original suggestion \cite{es00} for the sequence $(c_k)_{k\geq 0}$ was $c_k=1/k!$,
which puts more emphasis on shorter closed walks and ensures the convergence, 
since from~(\ref{eq-spectral-moment})
$$
f_c(G) = \sum_{k\geq 0} \frac{M_k(G)}{k!} 
       = \sum_{i=1}^n \sum_{k\geq 0} \frac{\lambda_i^k(G)}{k!} 
       = \sum_{i=1}^n e^{\lambda_i(G)}.
$$
This so-called Estrada index $EE(G)=\sum_{i=1}^n e^{\lambda_i}$ 
has been initially applied in measuring the degree of protein folding \cite{es00,es02,es04},
the centrality of complex networks \cite{es07} and the branching of molecular graphs \cite{es06,gufumagl07}.
It has been steadily gaining popularity in mathematical community,
as Zentralblatt now reports more than a hundred research articles on the Estrada index.


Properties of the $\preceq$ order have been studied in a few earlier papers.
Ili\'c and the author~\cite{ilst10} provided
an analogue of the Li-Feng lemma~\cite{life79} (see also \cite[Theorem 6.2.2]{cvrosi97}),
that represents a basic tool in dealing with the $\preceq$ order.
\begin{lemma}[\cite{ilst10}]
\label{le-li-feng}
Let $u$ be a vertex of a connected graph~$G$ with at least one edge.
For nonnegative integers $p$ and~$q$,
let $G(u;p,q)$ denote the graph obtained from~$G$ 
by attaching two pendent paths of lengths $p$ and $q$, respectively, at $u$.
If $p\geq q+2$, then 
$$
G(u;p,q)\prec G(u;p-1,q+1).
$$
\end{lemma}

Csikv\'ari~\cite{cs10} further introduced a generalized tree shift $GTS$,
that generalizes a transformation introduced earlier by Kelmans~\cite{ke81},
and showed that for any tree~$T$ with $T\not\cong GTS(T)$ holds $T\prec GTS(T)$.
In the induced poset of the generalized tree shift,
the path is the unique minimal, while the star is the unique maximal element,
so that these two trees have, respectively, 
the smallest and the largest number of closed walks of any given length among trees with $n$~vertices.
Note that the latter result can be shown using Lemma~\ref{le-li-feng} as well.

Bollob\'as and Tyomkyn~\cite{boty12} extended Csikv\'ari's result to show that
the generalized tree shift also increases the number of all walks,
showing that the path and the star have, respectively,
the smallest and the largest number of all walks of any given length among trees with $n$~vertices.

Andriantiana and Wagner~\cite{anwa13} studied trees with a given degree sequence,
and proved that the so-called greedy trees have maximum number of closed walks of any given length among such trees,
although they do not have to be unique such trees.
They further showed that, if a degree sequence~$D_1$ majorizes another degree sequence~$D_2$,
then the greedy tree for $D_1$ has more closed walks of any given length than the greedy tree for $D_2$,
which implied a proof of the conjecture of Ili\'c and the author~\cite{ilst10}
about trees with the maximum number of closed walks
among trees with a given number of vertices and maximum vertex degree.

Let us now define the class of trees that we will mostly study here.
Let $P_n$ denote a path on $n$~vertices.

\begin{definition}
A graph~$G$ is a {\em starlike tree}
if for some positive integers $a_1,\dots,a_k$, $k\geq 3$,
it can be obtained from the union of paths $\bigcup_{i=1}^k P_{a_i+1}$
by identifying one end of each path to a single vertex~$u$,
so that $G-u=\bigcup_{i=1}^k P_{a_i}$.
Such starlike tree is denoted shortly as $S(a_1,\dots,a_k)$.
\end{definition}

The vertex~$u$ from the above definition is called the center of~$S(a_1,\dots,a_k)$,
while the constituting paths, whose lengths are $a_1,\dots,a_k$, are called branches.
The numbers $a_1,\dots,a_k$ form a partition of $n-1$,
and it is usual to order the lengths as $a_1\leq\cdots\leq a_k$.
We can extend the above definition to $k\leq 2$ as well, 
but both $S(n-1)$ and $S(l,n-l-1)$, $1\leq l\leq n-2$, are then
trivial starlike trees that are isomorphic to the path~$P_n$.

Successive application of Lemma~\ref{le-li-feng} directly provides 
an initial $\prec$ ordering of trees on $n$~vertices for $n\in\N$:
$$
P_n \prec S(1,1,n-3) \prec S(1,2,n-4) \prec \cdots 
    \prec S\left(1,\left\lfloor\frac{n-2}2\right\rfloor,\left\lceil\frac{n-2}2\right\rceil\right).
$$
A natural question is then {\em how far this initial ordering extends as a linear order?}
The answer is actually not too far,
as the presence of just two vertices of degree at least three
may lead to pairs of $\preceq$ incomparable graphs,
an example of which is depicted in~Fig.~\ref{fig-incomparable}.
\begin{figure}[ht]
\begin{center}
\includegraphics[scale=0.3]{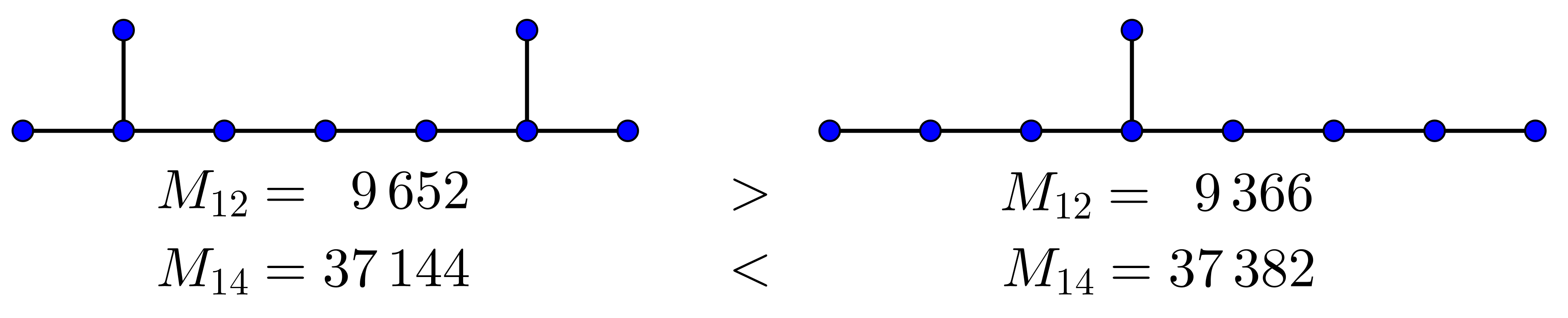}
\end{center}
\vspace{-12pt}
\caption{A pair of graphs incomparable by~$\preceq$.}
\label{fig-incomparable}
\end{figure}
Hence the most that one could expect is that 
$\preceq$ is a linear order among starlike trees,
and that is exactly what we will prove here. 
Moreover, our main result gives an easy way to compare starlike trees by~$\preceq$ 
using their branch lengths only, and without calculating their numbers of closed walks.
Recall now the definition of the shortlex order~\cite{si12}:
for two finite number sequences $a=(a_1,\dots,a_k)$ and $b=(b_1,\dots,b_l)$
we have $a<^\mathrm{lex}b$ 
if either $k<l$ or, 
when $k=l$, $a_i<b_i$ holds for the smallest index~$i$ at which the two sequences differ.
We can now state our main result.

\begin{theorem}
\label{th-main}
Let $1\leq a_1\leq\cdots\leq a_k$, $k\geq 3$, and $1\leq b_1\leq\cdots\leq b_l$, $l\geq 3$,
be two partitions of a natural number~$n$.
Then 
$$
S(a_1,\dots,a_k) \prec S(b_1,\dots,b_l)
\quad\Leftrightarrow\quad
(a_1,\dots,a_k)<^{\mathrm{lex}}(b_1,\dots,b_l).
$$
\end{theorem}

Before we start proving the main theorem, let us demonstrate its usefulness on an example.
Following the method described in~\cite[Section 2.2]{st15},
it is easy to show that for fixed~$k$
the spectral radius of $S(a_1,\dots,a_k)$ tends to $\frac{k}{\sqrt{k-1}}$,
when all the branch lengths $a_1,\dots,a_k$ independently tend to infinity.
Eigenvalue calculations in Octave, which use LAPACK routines, for example,
yield numerically indistinguishable values
$$
\lambda_1(S(80,90,100)),\ \lambda_1(S(85,90,95)),\ \lambda_1(S(90,90,90))\approx 2.12132034355964,
$$
and also
$$
EE(S(80,90,100)),\ EE(S(85,90,95)),\ EE(S(90,90,90))\approx 616.507916871363,
$$
while we know from the order isomorphism of $\prec$ and~$<^{\mathrm{lex}}$, as announced above,
that the spectral radii are ordered as
$$
\lambda_1(S(80,90,100))\leq \lambda_1(S(85,90,95))\leq \lambda_1(S(90,90,90))
$$
and Estrada indices as
$$
EE(S(80,90,100)) < EE(S(85,90,95)) < EE(S(90,90,90)).
$$
The weak inequality between spectral radii above is implied by the appearance of limit in~(\ref{eq-lambda-limit}).
The strict inequality between spectral radii of these starlike trees holds as well,
as Oliveira et al.~\cite{olsttr18} have recently used the Jacobs-Trevisan diagonalization algorithm~\cite{jatr11}
to show that also
\begin{equation}
\label{eq-lex-lambda}
(a_1,\dots,a_k)<^{\mathrm{lex}}(b_1,\dots,b_l)
\quad\Leftrightarrow\quad
\lambda_1(S(a_1,\dots,a_k))<\lambda_1(S(b_1,\dots,b_l)).
\end{equation}
Note that Oboudi dealt with the same problem in~\cite{ob18}, 
where he showed that if $(a_1,\dots,a_k)$ is majorized by $(b_1,\dots,b_l)$
then $\lambda_1(S(a_1,\dots,a_k))\leq\lambda_1(S(b_1,\dots,b_l))$.
However, Oboudi proved weak inequalities only and 
did not provide a complete characterization as in~(\ref{eq-lex-lambda}).

Structure of the paper is as follows.
In Section~\ref{sc-2} we briefly discuss properties of the shortlex order
and use the analogue of the Li-Feng lemma to cover a simple part of the proof.
In Section~\ref{sc-3} we present a walk embedding that settles
the case of pairs of starlike trees with different numbers of branches.
Finally, Section~\ref{sc-4} contains the main part of the proof
that deals with the remaining case when Lemma~\ref{le-li-feng} is not applicable.

\section{Shortlex order of partitions and 
         the case with only two changing parts}
\label{sc-2}

A slight modification of the Hindenburg's iterative algorithm for generating partitions~\cite{hi79,knxx},
which consists in reversing the list of parts from nonincreasing into nondecreasing order,
shows that two partitions of~$n$:
$$
(a_1,\dots,a_k)\mbox{ with }a_1\leq\cdots\leq a_k 
\quad\mbox{and}\quad
(b_1,\dots,b_l)\mbox{ with }b_1\leq\cdots\leq b_l
$$
are consecutive in the shortlex order if one of the following cases occur:

\medskip\noindent
{\bf Case I:}\quad 
$l=k$, $a_{k-1}\leq a_k-2$.
Then 
$$
(b_1,\dots,b_{k-2},b_{k-1},b_k)=(a_1,\dots,a_{k-2},a_{k-1}+1,a_k-1).
$$

\medskip\noindent
{\bf Case II:}\quad 
$l=k$, $a_j\leq a_k-2$ for $1\leq j\leq k-2$ 
and $a_k-1\leq a_t\leq a_k$ for $t=j+1,\dots,k-1$.
Then 
\begin{align*}
(b_1,\dots,b_{j-1}, & b_j,   \dots, b_{k-1}, b_k)= \\
(a_1,\dots,a_{j-1}, & a_j+1, \dots, a_j+1,   \sum_{t=j}^{k} a_t - (k-j)(a_j+1)).
\end{align*}

\medskip\noindent
{\bf Case III:}\quad
$l=k+1$ and there exists $j$ such that
$$
a_1=\dots=a_j=\left\lfloor\frac nk\right\rfloor,
\quad
a_{j+1}=\dots=a_k=\left\lceil\frac nk\right\rceil,
$$
while
$$
b_1=\dots=b_k=1,
\quad 
b_{k+1}=n-k.
$$

Since $\leq^{\mathrm{lex}}$ is a linear order on the finite set of partitions of~$n$,
in order to prove Theorem~\ref{th-main}
it is enough to show that
\begin{equation}
\label{eq-consecutive}
S(a_1,\dots,a_k)\prec S(b_1,\dots,b_l)
\end{equation}
holds whenever $(b_1,\dots,b_l)$ is a successor of $(a_1,\dots,a_k)$ in the shortlex order,
according to the cases listed above.

To prove~(\ref{eq-consecutive}) for Case I in which only two parts change their values, 
we can use Lemma~\ref{le-li-feng} directly.

\begin{proposition}
\label{pr-Case-I}
Let $1\leq a_1\leq\cdots\leq a_k$, $k\geq 3$, be a partition of~$n$ 
such that $a_{k-1}\leq a_k-2$. Then
$$
S(a_1,\dots,a_{k-2},a_{k-1},a_k)\prec S(a_1,\dots,a_{k-2},a_{k-1}+1,a_k-1).
$$
\end{proposition}

\begin{proof}
In order to apply Lemma~\ref{le-li-feng}, simply observe that 
$S(a_1,\dots,a_{k-2},a_{k-1},a_k)$ and $S(a_1,\dots,a_{k-2},a_{k-1}+1,a_k-1)$
can be viewed as $G(u;p,q)$ and $G(u;p-1,q+1)$, respectively,
for $G\cong S(a_1,\dots,a_{k-2})$, $q=a_{k-1}$ and $p=a_k$,
with the vertex $u$ determined as follows:
\begin{itemize}
\item if $k\geq 5$, $S(a_1,\dots,a_{k-2})$ is a proper starlike tree, 
      so that $u$ is its central vertex;
\item if $k=4$, $S(a_1,a_2)$ denotes the path $P_{a_1+a_2+1}$, 
      so that $u$ should be taken as a vertex at distance~$a_1$ from one of its leaves;
\item if $k=3$, $S(a_1)$ denotes the path $P_{a_1+1}$,
      so that $u$ should be taken as one of its leaves.
      \qedhere
\end{itemize}
\end{proof}

Case III is proved in the following section,
while Case II is proved in Section~\ref{sc-4}.

\section{Starlike trees with different degrees of central vertices}
\label{sc-3}

Here we use an interesting embedding of closed walks to prove Case III.
Although this case requests to compare
$S\left(\left\lfloor\frac nk\right\rfloor,\dots,\left\lfloor\frac nk\right\rfloor,
        \left\lceil\frac nk\right\rceil,\dots,\left\lceil\frac nk\right\rceil\right)$
to~$S(1,\dots,1,n-k)$ only,
the proof is applicable to all starlike trees with $k$~branches and we state it in that form.

\begin{proposition}
\label{pr-Case-III}
Let $1\leq a_1\leq\cdots\leq a_k$, $k\geq 3$, be an arbitrary partition of~$n$.
Then
$$
S(a_1,\dots,a_k)\prec S(\underbrace{1,\dots,1}_{k\ \mathrm{branches}},n-k).
$$
\end{proposition}

\begin{proof}
Let $u$ be the center of~$S(a_1,\dots,a_k)$.
For $i=1,\dots,k$, 
let $v_i$ be the neighbor of~$u$ in the branch of length~$a_i$ and
let $B_i$ denote the branch containing~$v_i$, 
but without the vertex~$u$ and the edge~$uv_i$
(see Fig.~\ref{fig-C}).
Hence $B_i$ has length $a_i-1$.
Similarly, let $u'$ be the center of $S(1,\dots,1,n-k)$ and
let $v'_1,\dots,v'_k$ be the neighbors of~$u'$ 
that form branches of length one.
Let $B'$ denote the subpath $S(1,\dots,1,n-k)-\{v'_1,\dots,v'_k\}$,
which contains~$u'$ and has length $n-k$. 

\begin{figure}
\begin{center}
\includegraphics[scale=0.5]{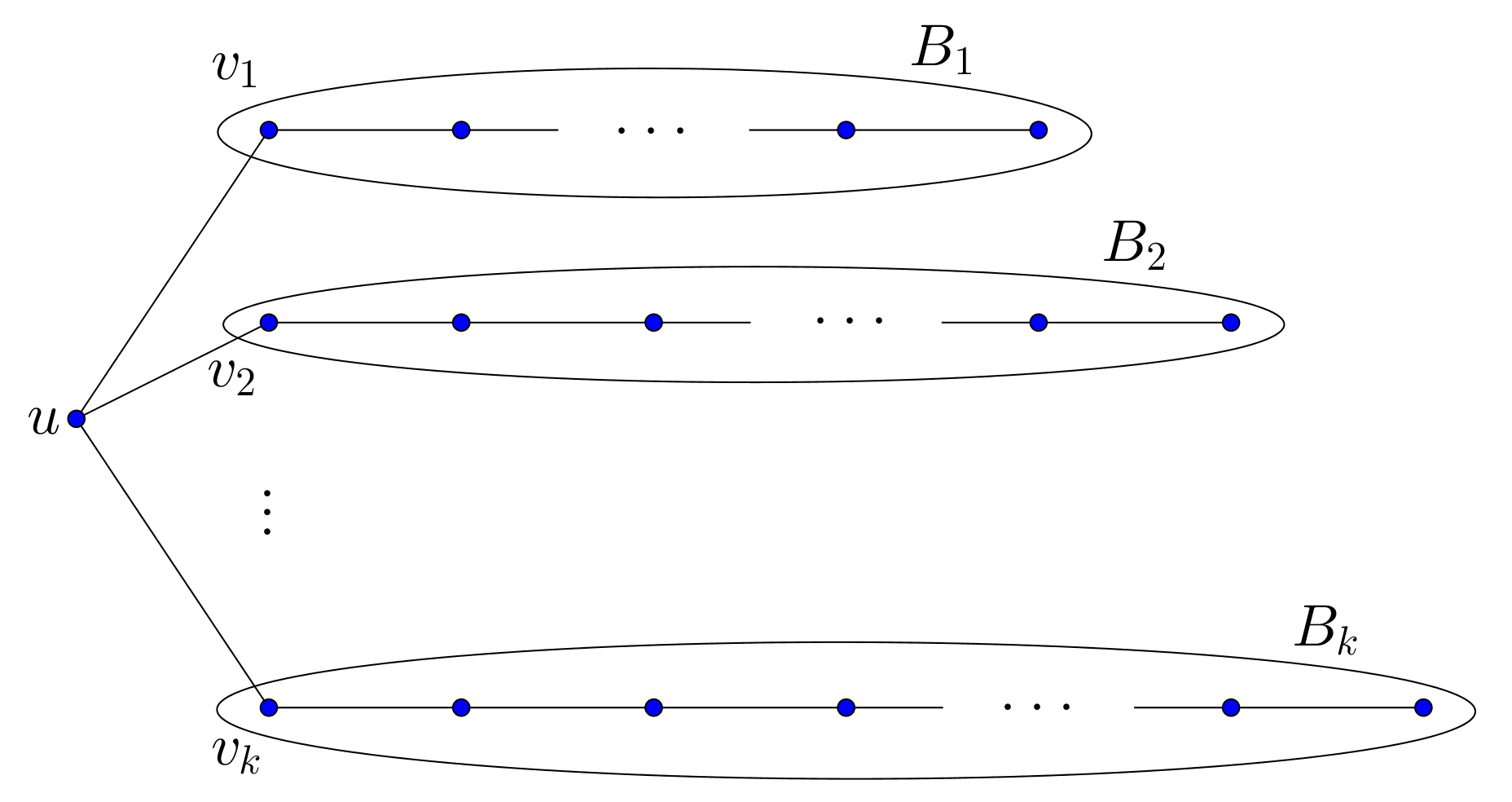}
\quad
\includegraphics[scale=0.5]{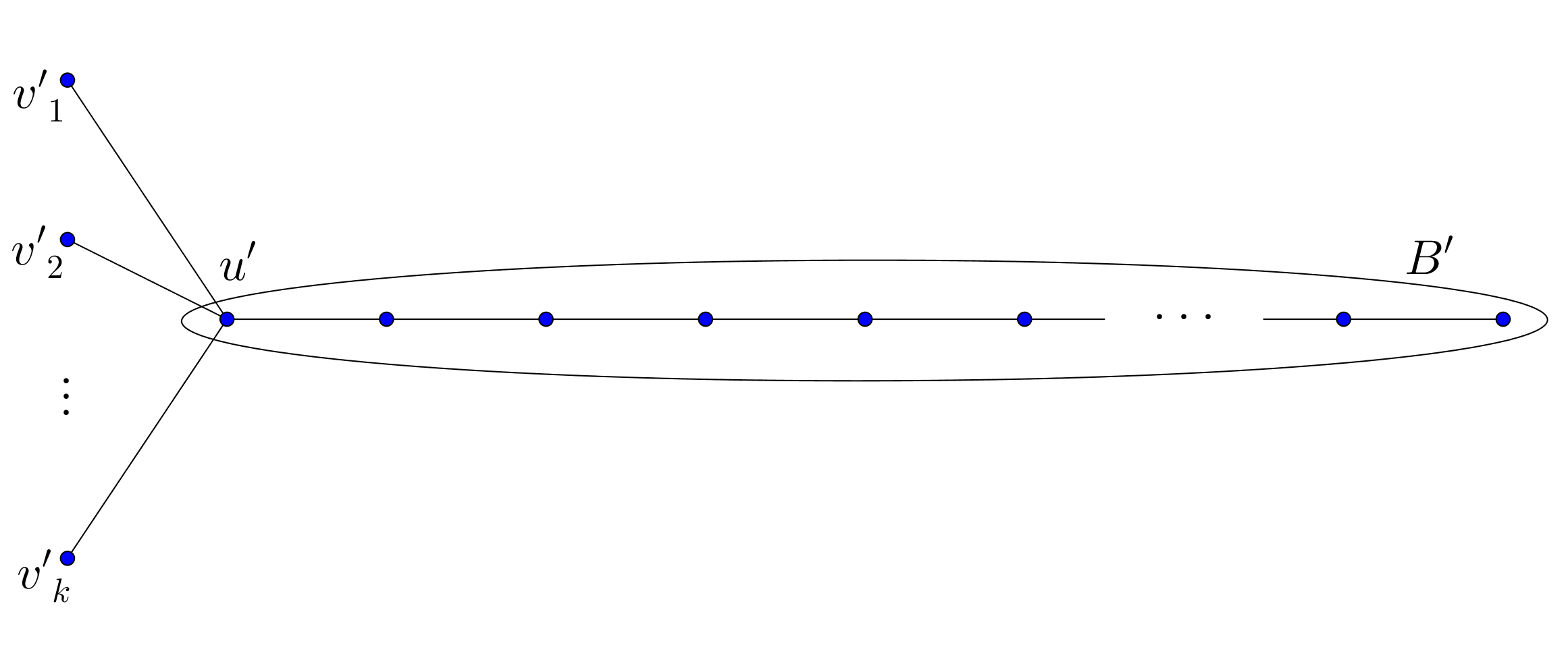}
\end{center}
\caption{Starlike trees $S(a_1,\dots,a_k)$ and $S(1,\dots,1,n-k)$.}
\label{fig-C}
\end{figure}

Closed walks of $S(a_1,\dots,a_k)$ may be classified into the following three types:

\medskip
{\em Type i)}\quad
closed walks that do not contain any of the edges $uv_1$,\dots,$uv_k$;

\medskip
{\em Type ii.a)}\quad
closed walks that start and end at~$u$;

\medskip
{\em Type ii.b)}\quad
closed walks that contain~$u$, but do not start at~$u$.

\medskip
We will now construct an injective embedding~$F$ 
from the set of closed $l$-walks of $S(a_1,\dots,a_k)$
into the set of closed $l$-walks of $S(1,\dots,1,n-k)$, 
according to the above walk types.

\medskip
{\em Type i)}\quad
Note that the total number $\sum_{i=1}^k (a_i-1)$ of edges in subpaths $B_1,\dots,B_k$
is equal to the number of edges in the subpath $B'$ of $S(1,\dots,1,n-k)$.
We can thus partition the edges of~$B'$ 
into edge-disjoint subpaths $B'_1,\dots,B'_k$
such that $B'_i$ has length $a_i-1$, $i=1,\dots,k$.
Let $g_i$ be an isometric embedding that maps~$B_i$ to~$B'_i$ for $i=1,\dots,k$.

Now, let $W\colon w_0,w_1,\dots,w_l$ be a closed $l$-walk of $S(a_1,\dots,a_k)$ that
does not contain any of the edges $uv_1,\dots,uv_k$.
This implies that $W$ fully belongs to a subpath~$B_i$ for some $1\leq i\leq k$.
We now set $F(W)$ to be the walk $g_i(w_0),g_i(w_1),\dots,g_i(w_l)$,
so that $W$~is essentially translated from $B_i$ to~$B'_i$.
Hence $F$ bijectively maps
closed $l$-walks of $S(a_1,\dots,a_k)$ 
that fully belong to some~$B_i$, $1\leq i\leq k$,
into closed $l$-walks of $S(1,\dots,1,n-k)$
that fully belong to the corresponding~$B'_i$.

\medskip
{\em Type ii.a)}\quad
For $i=1,\dots,k$
let $h_i$ be an isometric embedding that maps $B_i$ to~$B'$ such that $h_i(v_i)=u'$.
Thus each $B_i$ is mapped by~$h_i$ to the initial part of~$B'$ of the same length.

Now, let $W$ be a closed $l$-walk of $S(a_1,\dots,a_k)$ that starts and ends at~$u$.
For some~$m$ and $i_1,\dots,i_m\in\{1,\dots,k\}$
it has the form
$$
W\colon u,W_1,u,W_2,u,\dots,u,W_m,u
$$
where $W_j$ is a subwalk fully contained in~$B_{i_j}$, $j=1,\dots,m$.
Hence $W_j$ starts and ends with~$v_{i_j}$,
so that $h_{i_j}(W_j)$ is an isometric copy of~$W_j$ in~$B'$ that starts and ends at~$u'$.
Now set
$$
F(W)\colon u',v'_{i_1},h_{i_1}(W_1),v'_{i_2},h_{i_2}(W_2),\dots,v'_{i_m},h_{i_m}(W_m).
$$
Thus $F(W)$ is obtained 
by replacing the first edge~$uv_{i_j}$ and the last edge~$v_{i_j}u$ 
from each subwalk~$u,W_j,u$
with a pair of edges $u'v'_{i_j}$, $v'_{i_j}u'$ followed by a copy of~$W_j$ in~$B'$.
This enables easy reconstruction of~$W$ from~$F(W)$
as occurrences of vertices $v'_1$, \dots, $v'_l$ in~$F(W)$ serve to
extract subwalks $h_{i_1}(W_1)$, \dots, $h_{i_m}(W_m)$ from~$F(W)$.
Each subwalk~$h_{i_j}(W_j)$ is isometrically mapped back to the subwalk~$W_j$
in the branch of $S(a_1,\dots,a_k)$ 
that corresponds to its leading $v'_{i_j}$ vertex,
which shows injectivity of~$F$ in this case.

\medskip
{\em Type ii.b)}\quad
Finally, let $W$ be a closed $l$-walk of~$S(a_1,\dots,a_k)$
that contains~$u$, but does not start at it.
It has the form
$$
W\colon W_0,u,W_1,u,\dots,u,W_m,u,W_{m+1}
$$
for some $m$ and $i_0,\dots,i_{m+1}\in\{1,\dots,k\}$
such that a subwalk~$W_j$ is fully contained in~$B_{i_j}$, $j=0,\dots,m+1$.
Note that since $W$ is closed, 
$W_0$ and $W_{m+1}$ belong to the same branch $B_{i_0}=B_{i_{m+1}}$.
Similar to the previous case, set
$$
F(W)\colon h_{i_0}(W_0),v'_{i_1},
           h_{i_1}(W_1),v'_{i_2},\dots,v'_{i_m},
           h_{i_m}(W_m),v'_{i_{m+1}}, 
           h_{i_{m+1}}(W_{m+1}).
$$
Again, occurrences of vertices $v'_1,\dots,v'_k$ in~$F(W)$ serve
to delimit the subwalks $h_{i_0}(W_0)$, \dots, $h_{i_{m+1}}(W_{m+1})$ in~$F(W)$.
One can then reconstruct the original walk~$W$
by isometrically mapping back each subwalk $h_{i_j}(W_j)$ 
to the subwalk~$W_j$ in the branch of~$S(a_1,\dots,a_k)$ 
that corresponds to its leading $v'_{i_j}$ vertex,
where $h_{i_0}(W_0)$ is mapped back to the same branch~$B_{i_{m+1}}$ as $h_{i_{m+1}}(W_{m+1})$.
Hence $F$ is injective in this case as well.

It is easily seen that $F$ is actually injective over its whole domain,
as the closed $l$-walks of $S(a_1,\dots,a_k)$ of different types above 
get mapped by~$F$ to disjoint subsets of $l$-walks of~$S(1,\dots,1,n-k)$:
\begin{itemize}
\item
walks of type i) are mapped to closed walks 
that do not contain any of the edges~$u'v'_j$,
\item
walks of type ii.a) are mapped to closed walks
that contain at least one of the edges~$u'v'_j$ and start at~$u'$,
while
\item
walks of type ii.b) are mapped to closed walks
that contain at least one of the edges $u'v'_j$ and do not start at~$u'$.
\end{itemize}

On the other hand, $F$ is not surjective
as it embeds closed walks either into disjoint parts of~$B'$
or into initial parts of~$B'$ together with the edges $u'v'_1,\dots,u'v'_k$.
Hence no closed walk of~$S(a_1,\dots,a_k)$ may be mapped by~$F$
to a closed walk of~$S(1,\dots,1,n-k)$ that
contains two vertices of~$B'$ that are at distance at least $\max_i a_i$ apart.
\end{proof}

\section{The case with at least three changing parts}
\label{sc-4}

In this section we settle the remaining Case II.
\begin{proposition}
\label{pr-Case-II}
Let $1\leq a_1\leq\cdots\leq a_k$, $k\geq 3$, be a partition of $n$
such that $a_j\leq a_k-2$ for some $j\leq k-2$ and
$a_t\in\{a_k-1,a_k\}$ for $t=j+1,\dots,k-1$.
Then
\begin{eqnarray*}
&     & S(a_1,\dots,a_{j-1},a_j,\dots,a_k) \\
&\prec& S(a_1,\dots,a_{j-1},a_j+1,\dots,a_j+1,\sum_{t=j}^k a_t-(k-j)(a_j+1)).
\end{eqnarray*}
\end{proposition}

To simplify notation, denote by $a=a_j$ and $b=a_k-1$
and let $p$ and $q$ be such that
$a_{j+1}=\dots=a_{j+p}=b$, while $a_{j+p+1}=\dots=a_{j+p+q}=b+1$,
where $j+p+q=k$.
Further, let
$$
f=\sum_{t=j}^k a_t-(k-j)(a_j+1)
=(p+q-1)(b-a) + b-p.
$$
The statement to be proved now becomes
\begin{eqnarray}
\nonumber
&& S(a_1,\dots,a_{j-1},a,\underbrace{b,\dots,b}_{p\ \mathrm{times}},
                         \underbrace{b+1,\dots,b+1}_{q\ \mathrm{times}}) \\
\label{eq-Case-III}
&\prec& S(a_1,\dots,a_{j-1},\underbrace{a+1,\dots,a+1}_{p+q\ \mathrm{times}},f).
\end{eqnarray}
Note that $f\geq b+1$ unless $b-a=1$ and $q=1$,
%
%
%
%
in which case $f=b$ and~(\ref{eq-Case-III}) reduces to
$$
S(a_1,\dots,a_{j\!-\!1},a,\underbrace{a\!+\!1,\dots,a\!+\!1}_{p\ \mathrm{times}},
                          a+2)\prec
S(a_1,\dots,a_{j\!-\!1},\underbrace{a\!+\!1,\dots,a\!+\!1}_{p+2\ \mathrm{times}}),
$$
that follows directly from Lemma~\ref{le-li-feng}
applied to the branches of lengths $a$ and~$a+2$ in the first starlike tree.
Hence we can assume that $f\geq b+1$ in the sequel.

To prove~(\ref{eq-Case-III}) in the general case,
we will need a lemma on closed walks in coalescences of graphs.
Recall that a {\em coalescence} of two vertex disjoint graphs $G$ and~$H$,
with respect to a vertex~$u$ of~$G$ and a vertex~$v$ of~$H$,
denoted by $G(u=v)H$, is obtained from the union of $G$ and~$H$
by identifying vertices $u$ and~$v$.

\begin{lemma}
\label{le-coalescence}
Let $G$, $H_1$ and $H_2$ be three vertex disjoint graphs,
and let $u$ be a vertex of~$G$, $v_1$ a vertex of $H_1$ and $v_2$ a vertex of~$H_2$.
If $H_1\preceq H_2$ and $M_k(H_1,v_1)\leq M_k(H_2,v_2)$ for all $k\geq 0$,
then
$$
G(u=v_1)H_1 \preceq G(u=v_2)H_2.
$$
Strict inequality holds if either $H_1\prec H_2$ or $M_k(H_1,v_1)<M_k(H_2,v_2)$ for some~$k$.
\end{lemma}

\begin{proof}
Since $M_k(H_1,v_1)\leq M_k(H_2,v_2)$ for each $k\geq 0$,
we can choose an injective mapping~$F^c_k$
from the set of closed walks in~$H_1$ of length~$k$ starting from~$v_1$
to the set of closed walks in~$H_2$ of length~$k$ starting from~$v_2$.
As the domains, as well as codomains, of~$F^c_k$ 
are mutually disjoint for distinct values of~$k$,
the union of these mappings $F^c=\cup_{k\geq 0} F^c_k$
will injectively map
the set of closed walks in~$H_1$ starting from~$v_1$
to the set of closed walks in~$H_2$ starting from~$v_2$.
Note that for each closed walk~$W$ in $H_1$ starting from~$v_1$,
its image $F^c(W)$ has the same length as~$W$.

In a similar way, we see that due to $H_1\preceq H_2$ 
we can choose an injective mapping~$F^a$ 
from the set of closed walks in~$H_1$
to the set of closed walks in~$H_2$,
such that for any closed walk~$W$ in $H_1$
its image $F^a(W)$ has the same length as~$W$.

Now we will construct an injective mapping~$I_k$
from the set of closed $k$-walks of $G(u=v_1)H_1$
to the set of closed $k$-walks of $G(u=v_2)H_2$.
Let $W$ be an arbitrary closed $k$-walk of $G(u=v_1)H_1$.
If $W$ contains edges from~$G$ only,
then we can set $I_k(W)=W$.
If, on the other hand, $W$ contains edges from $H_1$ only,
then we set $I_k(W)=F^a(W)$.

The more interesting cases arise when $W$ contains edges from both $G$ and~$H_1$.
The vertex $u=v_1$ then serves as the gate through which
the walk can pass from $G$ to~$H_1$ and vice versa, and 
$W$ has to contain at least two appearances of~$u$.
Let $m$ and $W_0,\dots,W_{m}$ be such that $W$ has the form
\begin{equation}
\label{eq-W-form}
W\colon W_0, u, W_1, u, \dots, u, W_{m-1}, u, W_m
\end{equation}
and that none of $W_0,\dots,W_m$ contains~$u$,
so that each $W_i$ belongs fully to either $G$ or~$H_1$.
We will now choose $W'_0,\dots,W'_m$ so that $I_k$ maps $W$ to
\begin{equation}
\label{eq-Ik-form}
I_k(W)\colon W'_0, u, W'_1, u, \dots, u, W'_{m-1}, u, W'_m.
\end{equation}
For $i\neq 0, m$,
if $W_i$ belongs to~$G$, it is mapped to itself so that $W'_i=W_i$,
while if $W_i$ belongs to~$H_1$,
determine $W'_i$ from the image $F^c(v_1,W_i,v_1)=v_2,W'_i,v_2$.

For $W_0$ and $W_m$,
if they are both empty or both belong to~$G$,
then we can also set $W'_0=W_0$ and $W'_m=W_m$.

In the case that $W_0$ and $W_m$ both belong to~$H_1$,
then let $Y_0$ and $Y_m$ be such that
$W_0\colon w, Y_0$ and $W_m\colon Y_m,w$,
where $w$ is the first and the last vertex of the closed walk~$W$.
Then $W^*\colon v_1,Y_m,w,Y_0,v_1$ is a closed walk in~$H_1$ starting at~$v_1$.
The image $F^c(W^*)$ is a closed walk in~$H_2$ starting at $v_2$ 
that has the same length as $W^*$.
Now let $Y'_0$ and $Y'_m$ be such that
$$
F^c(W^*)\colon v_2,Y'_m, w', Y'_0,v_2
$$
and that the walks $v_1,Y_m,w$ and $v_2,Y'_m,w'$, 
as well as $w,Y_0,v_1$ and $w',Y'_0,v_2$, have the same length.
Finally, set $W'_0\colon w',Y'_0$ and $W'_m\colon Y'_m,w'$.

One can see from the previous construction that
$I_k(W)$ is a closed walk in $G(u=v_2)H_2$
that has the same length as $W$, i.e., $k$.
The mapping $I_k$ is injective,
as we can easily reconstruct $W$ uniquely from $W'=I_k(W)$:
\begin{itemize}
\item if $W'$ contains edges from $G$ only, then $W=W'$;
\item if $W'$ contains edges from $H_2$ only, then $W=(F^a)^{-1}(W')$;
\item if $W'$ contains edges from both $G$ and~$H_2$,
      then $W'$ can be partitioned into subwalks $W'_0,\dots,W'_m$ according to~(\ref{eq-Ik-form}).
      Each subwalk~$W'_i$ then yields an appropriate subwalk~$W_i$ using 
      either identity map if $W'_i$ belongs to $G$ or $(F^c)^{-1}$ if $W'_i$ belongs to~$H_2$
      (with appropriate recombination of $W'_0$ and $W'_m$ when they both belong to~$H_2$),
      and $W$ can be obtained by combining subwalks $W_0,\dots,W_m$ according to~(\ref{eq-W-form}).
\end{itemize}

If either $H_1\prec H_2$ or $M_k(H_1,v_1)<M_k(H_2,v_2)$ for some~$k$,
then either $F^a$ or $F^c$ is not surjective,
so that $G(u=v)H_1\prec G(u=v_2)H_2$ holds.
\end{proof}

Note that the previous lemma can also be extended to multiple coalescences of graphs,
considered in~\cite{st15b}, that are obtained by identifying several pairs of vertices at once.

In order to apply Lemma~\ref{le-coalescence},
let $u$ be the central vertex of $S(a_1,\dots,a_{j-1})$
(or an appropriately chosen vertex when $j\leq 3$,
 see the proof of Proposition~\ref{pr-Case-I}),
$v_1$ the central vertex of 
$S(a,\underbrace{b,\dots,b}_{p\ \mathrm{times}}, \underbrace{b+1,\dots,b+1}_{q\ \mathrm{times}})$,
and $v_2$ the central vertex of $S(\underbrace{a+1,\dots,a+1}_{p+q\ \mathrm{times}},f)$.
Then $S(a_1,\dots,a_{j-1},a,\underbrace{b,\dots,b}_{p\ \mathrm{times}},
                         \underbrace{b+1,\dots,b+1}_{q\ \mathrm{times}})$
and $S(a_1,\dots,a_{j-1},\underbrace{a+1,\dots,a+1}_{p+q\ \mathrm{times}},f)$
can be considered as the coalescences:
\begin{eqnarray*}
&& S(a_1,\dots,a_{j-1},a,\underbrace{b,\dots,b}_{p\ \mathrm{times}},
                      \underbrace{b+1,\dots,b+1}_{q\ \mathrm{times}}) \\
&& \cong
   S(a_1,\dots,a_{j-1})
   (u=v_1)
   S(a,\underbrace{b,\dots,b}_{p\ \mathrm{times}},
       \underbrace{b+1,\dots,b+1}_{q\ \mathrm{times}}), \\
&& S(a_1,\dots,a_{j-1},\underbrace{a+1,\dots,a+1}_{p+q\ \mathrm{times}},f) \\
&& \cong
   S(a_1,\dots,a_{j-1})
   (u=v_2)
   S(\underbrace{a+1,\dots,a+1}_{p+q\ \mathrm{times}},f).
\end{eqnarray*}
By Lemma~\ref{le-coalescence},
the statement~(\ref{eq-Case-III}) will follow
from the following two statements:
\begin{eqnarray*}
   S(a,\underbrace{b,\dots,b}_{p\ \mathrm{times}},
       \underbrace{b+1,\dots,b+1}_{q\ \mathrm{times}})
&\prec& S(\underbrace{a+1,\dots,a+1}_{p+q\ \mathrm{times}},f), \\
M_k(S(a,\underbrace{b,\dots,b}_{p\ \mathrm{times}},
       \underbrace{b+1,\dots,b+1}_{q\ \mathrm{times}}),v_1)
&\leq& M_k(S(\underbrace{a+1,\dots,a+1}_{p+q\ \mathrm{times}},f),v_2),
\end{eqnarray*}
for all $k\geq 0$,
which we prove in the next two subsections.

\subsection{All closed walks in 
$S(a,b,\dots,b,b+1,\dots,b+1)$ and $S(a+1,\dots,a+1,f)$}

\begin{proposition}
\label{pr-all-closed-walks}
For arbitrary positive integers $a,b$ and~$q$ and nonnegative integer~$p$
such that $a<b$ and $p+q\geq 2$, except the case when $b=a+1$ and $q=1$,
let $f=(p+q-1)(b-a)+b-p$. Then
$$
S(a,\underbrace{b,\dots,b}_{p\ \mathrm{times}},
    \underbrace{b+1,\dots,b+1}_{q\ \mathrm{times}})
\prec S(\underbrace{a+1,\dots,a+1}_{p+q\ \mathrm{times}},f).
$$
\end{proposition}

\begin{proof}
The excluded case $b=a+1$ and $q=1$ 
(that has been dealt with by Lemma~\ref{le-li-feng} right after (\ref{eq-Case-III}) was stated)
yields $f=b$,
while all other cases imply that $f\geq b+1$.

We prove the inequality between the numbers of closed walks in these starlike trees
by factoring their characteristic polynomials,
that reveals a common part of their spectra
that cancels out after taking a difference of their spectral moments,
and by suitably interpreting 
the difference of the numbers of closed walks of appropriate subgraphs.

For a polynomial $Q(\lambda)$,
let $Sp(Q)$ denote the family of roots of~$Q$.
For a family~$L$ and a natural number~$m$,
let $mL$ denote the family in which every element of~$L$ is repeated $m$~times.
Then for a graph~$G$,
let $P(G,\lambda)=\det(\lambda I-A(G))$
denote the characteristic polynomial of the adjacency matrix of~$G$,
whose family of roots forms the spectrum consisting of all eigenvalues of~$G$,
that we will denote directly as $Sp(G)$.
To simplify notation, 
we will further abbreviate $P(P_n,\lambda)$ by $\bm{P}_n$.

\begin{lemma}
\label{le-factoring}
For $q\geq 2$ and arbitrary positive integers $c$ and~$d$,
\begin{equation}
\label{eq-factoring}
P(S(c,\underbrace{d,\dots,d}_{q\ \mathrm{times}}),\lambda)
=\bm{P}_d^{q-1}\left(\bm{P}_{c+d+1} - (q-1)\bm{P}_c\bm{P}_{d-1}\right).
\end{equation}
\end{lemma}

\begin{proof}[Proof of Lemma~\ref{le-factoring}]
Schwenk~\cite{sc74} proved that if $uv$ is a cut edge of~$G$, then
$$
P(G,\lambda)= P(G-uv,\lambda)-P(G-u-v,\lambda)
$$
where $G-uv$ means removal of edge~$uv$ only,
while $G-u-v$ means removal of vertices $u$ and~$v$ and all incident edges from~$G$.
Applying this result to an edge incident with the center of~$S(c,d,d)$
that belongs to one of the branches of length~$d$,
we obtain
\begin{eqnarray*}
P(S(c,d,d),\lambda)
 &=& \bm{P}_{c+d+1}\bm{P}_d-\bm{P}_c\bm{P}_d\bm{P}_{d-1} \\
 &=& \bm{P}_d\left(\bm{P}_{c+d+1}-\bm{P}_c\bm{P}_{d-1}\right),
\end{eqnarray*}
which proves the lemma in the case~$q=2$.
Taking this case as the basis of induction,
and assuming that (\ref{eq-factoring}) is valid for some $q-1\geq 2$,
we obtain by applying Schwenk's formula to 
an edge incident with the center of~$S(c,d,\dots,d)$
that belongs to one of the branches of length~$d$,
\begin{eqnarray*}
P(S(c,\underbrace{d,\dots,d}_{q\ \mathrm{times}}),\lambda)
 &=& \bm{P}_d P(S(c,\underbrace{d,\dots,d}_{q-1\ \mathrm{times}}),\lambda)
    -\bm{P}_{d-1}\bm{P}_c\bm{P}_d^{q-1} \\
 &=& \bm{P}_d\bm{P}_d^{q-2}\left(\bm{P}_{c+d+1}-(q-2)\bm{P}_c\bm{P}_{d-1}\right)
    -\bm{P}_{d-1}\bm{P}_c\bm{P}_d^{q-1} \\
 &=& \bm{P}_d^{q-1}\left(\bm{P}_{c+d+1}-(q-1)\bm{P}_c\bm{P}_{d-1}\right),
\end{eqnarray*}
which proves (\ref{eq-factoring}) for~$q$ as well.
\end{proof}

From Lemma~\ref{le-factoring} we now get
\begin{align*}
P(S(a,\underbrace{b+1,\dots,b+1}_{p+q\ \mathrm{times}}),\lambda)
 &= \bm{P}_{b+1}^{p+q-1}\left(\bm{P}_{a+b+2}-(p+q-1)\bm{P}_a\bm{P}_b\right), \\
P(S(\underbrace{a+1,\dots,a+1}_{p+q\ \mathrm{times}},b),\lambda)
 &= \bm{P}_{a+1}^{p+q-1}\left(\bm{P}_{a+b+2}-(p+q-1)\bm{P}_a\bm{P}_b\right).
\end{align*}
A common factor of these characteristic polynomials means that
these two starlike trees share a part of the spectrum:
\begin{align*}
Sp(S(a,\underbrace{b\!+\!1,\dots,b\!+\!1}_{p+q\ \mathrm{times}}),\lambda)
 &= (p\!+\!q\!-\!1)Sp(\bm{P}_{b\!+\!1})
    \cup Sp\left(\bm{P}_{a+b+2}\!-\!(p\!+\!q\!-\!1)\bm{P}_a\bm{P}_b\right), \\
Sp(S(\underbrace{a\!+\!1,\dots,a\!+\!1}_{p+q\ \mathrm{times}},b),\lambda)
 &= (p\!+\!q\!-\!1)Sp(\bm{P}_{a\!+\!1})
    \cup Sp\left(\bm{P}_{a+b+2}\!-\!(p\!+\!q\!-\!1)\bm{P}_a\bm{P}_b\right),
\end{align*}
which cancels out if we subtract their spectral moments~(\ref{eq-spectral-moment}):
\begin{eqnarray}
\nonumber
&& M_k((S(a,\underbrace{b+1,\dots,b+1}_{p+q\ \mathrm{times}}))
  - M_k(S(\underbrace{a+1,\dots,a+1}_{p+q\ \mathrm{times}},b)) \\
\label{eq-moment-canceling}
&=& (p+q-1)\left(M_k(P_{b+1}) - M_k(P_{a+1})\right).
\end{eqnarray}

Interpretation of the difference of the numbers of closed walks of paths,
appearing in the previous equation,
is provided by the following lemma and its corollaries.
\begin{lemma}
\label{le-path-difference}
Let $c$ and $d$ be positive integers and let $u$ be a vertex of a graph~$G$
such that $G$ contains a path~$P_{c+1}$ as a proper subgraph, 
with $u$ as one of its leaves.
If $v$ is a leaf of a path~$P_{d+1}$ that is vertex disjoint from~$G$,
then
\begin{equation}
\label{eq-path-difference}
M_k(G(u=v)P_{d+1}) \geq M_k(G) + M_k(P_{c+d+1}) - M_k(P_{c+1}).
\end{equation}
Strict inequality holds above for all sufficienly large values of~$k$.
\end{lemma}

\begin{proof}[Proof of Lemma~\ref{le-path-difference}]
After identifying $u$ and $v$ in the coalescence~$G(u=v)P_{d+1}$,
the path $P_{d+1}$ contributes $d$~new edges to the coalescence.
Color the edges of $G(u=v)P_{d+1}$ such that
the edges of~$P_{d+1}$ are red,
the edges of the proper subgraph~$P_{c+1}$ of~$G$ are blue,
while the remaining edges of~$G$ are black.
Edge disjoint paths $P_{c+1}$ and~$P_{d+1}$ in the coalescence $P_{c+1}(u=v)P_{d+1}$ 
form a path $P_{c+d+1}$ consisting of a total of $c+d$ blue and red edges.
Now $M_k(G)$ represents closed $k$-walks of $G(u=v)P_{d+1}$
consisting of blue and black edges,
while $M_k(P_{c+d+1})-M_k(P_{c+1})$ represents closed $k$-walks
consisting of blue and red edges,
that contain at least one red edge.
Sets of closed $k$-walks that these two terms count are disjoint,
from which the inequality~(\ref{eq-path-difference}) follows directly.
Moreover, for all sufficiently large values of~$k$,
$G(u=v)P_{d+1}$ contains closed $k$-walks that contain both black and red edges,
which are not counted by either of the two terms on the right hand side,
so that strict inequality then holds in~(\ref{eq-path-difference}).
\end{proof}

Repeated application of this lemma yields the following corollaries.
\begin{corollary}
\label{co-path-difference-disjoint}
Let $c_1,\dots,c_l$ and $d_1,\dots,d_l$ be 
two sequences of positive integers for some $l\geq 1$,
and let $u$ be a vertex of a graph~$G$ 
such that $G$ contains a path $P_{\max\{c_1,\dots,c_l\}+1}$ as a proper subgraph,
with $u$ as one of its leaves.
If paths $P_{d_1+1},\dots,P_{d_l+1}$ are mutually vertex disjoint,
and also vertex disjoint from~$G$,
and if $v_i$ is a leaf of the path $P_{d_i+1}$ for $i=1,\dots,l$,
then
\begin{align*}
& M_k(G(u=v_1)P_{d_1+1}\cdots(u=v_l)P_{d_l+1}) \\
& \geq M_k(G) + \sum_{i=1}^l \left(M_k(P_{c_i+d_i+1}) - M_k(P_{c_i+1})\right).
\end{align*}
\end{corollary}

\begin{corollary}
\label{co-path-difference-sequential}
Let $c_1,\dots,c_l$ and $d_1,\dots,d_l$ be 
two sequences of positive integers for some $l\geq 1$.
Let $P_{d_1+1},\dots,P_{d_l+1}$ be vertex disjoint paths,
with $u_i$ as one and $v_i$ as another leaf of $P_{d_i+1}$ for $i=1,\dots,l$.
Let $G$ be a graph, vertex disjoint from all of $P_{d_1+1},\dots,P_{d_l+1}$,
and $v_0$ one of its vertices, such that 
$G$ contains a path $P_{c_1+1}$ as a proper subgraph with $v_0$ as one of its leaves.
If further $G(v_0=u_1)P_{d_1+1}\cdots(v_{i-2}=u_{i-1})P_{d_{i-1}+1}$
contains a path $P_{c_1+\cdots+c_i+1}$ with $v_i$ as one of its leaves for each $i=2,\dots,l$,
then
\begin{align*}
& M_k(G(v_0=u_1)P_{d_1+1}\cdots(v_{l-1}=u_l)P_{d_l+1}) \\
& \geq M_k(G) + \sum_{i=1}^l \left(M_k(P_{c_i+d_i+1}) - M_k(P_{c_i+1})\right).
\end{align*}
\end{corollary}

\begin{proof}[Proof of Corollary~\ref{co-path-difference-sequential}]
Repeated application of Lemma~\ref{le-path-difference} yields
\begin{align*}
& M_k(G(v_0\!\!=\!\!u_1)P_{d_1\!+\!1}\cdots(v_{l\!-\!1}\!\!=\!\!u_l)P_{d_l\!+\!1}) \\
& \geq M_k(G(v_0\!\!=\!\!u_1)P_{d_1\!+\!1}\cdots
                           (v_{l\!-\!2}\!\!=\!\!u_{l\!-\!1})P_{d_{l\!-\!1}\!+\!1})
 \!+\! \left(M_k(P_{c_l\!+\!d_l\!+\!1})\!-\!M_k(P_{c_l\!+\!1})\right) \\
& \geq M_k(G(v_0\!\!=\!\!u_1)P_{d_1\!+\!1}\cdots
                           (v_{l\!-\!3}\!\!=\!\!u_{l\!-\!2})P_{d_{l\!-\!2}\!+\!1})
 \!+\!\!\!\sum_{i=l-1}^l\!\!\left(M_k(P_{c_i\!+\!d_i\!+\!1})\!-\!M_k(P_{c_i\!+\!1})\right) \\
& \dots \\
& \geq M_k(G)
  \!+\!\!\!\sum_{i=1}^l \left(M_k(P_{c_i\!+\!d_i\!+\!1})\!-\!M_k(P_{c_i\!+\!1})\right).
& \qedhere  
\end{align*}
\end{proof}

Back to the trees 
$S(a,\underbrace{b,\dots,b}_{p\ \mathrm{times}},
     \underbrace{b+1,\dots,b+1}_{q\ \mathrm{times}})$
and
$S(\underbrace{a+1,\dots,a+1}_{p+q\ \mathrm{times}},f)$
that we are concerned with in Proposition~\ref{pr-all-closed-walks}, 
application of Corollary \ref{co-path-difference-disjoint} yields
\begin{align}
\nonumber
& M_k(S(a,\underbrace{b+1,\dots,b+1}_{p+q\ \mathrm{times}})) \\
\label{eq-walk-inequality-1}
& \geq M_k(S(a,\underbrace{b,\dots,b}_{p\ \mathrm{times}},
               \underbrace{b+1,\dots,b+1}_{q\ \mathrm{times}}))
  + p\left(M_k(P_{b+1})-M_k(P_b)\right),
\end{align}
while, recalling that $f=(p+q-1)(b-a)+b-p = b+(q-1)(b-a)+p(b-a-1)$, 
application of Corollary~\ref{co-path-difference-sequential} yields
\begin{align}
\nonumber
& M_k(S(\underbrace{a+1,\dots,a+1}_{p+q\ \mathrm{times}},f)) \\
\nonumber
& \geq M_k(S(\underbrace{a+1,\dots,a+1}_{p+q\ \mathrm{times}},b)) \\
\label{eq-walk-inequality-2}
& + (q-1)\left(M_k(P_{b+1})-M_k(P_{a+1})\right)
  + p\left(M_k(P_{b})-M_k(P_{a+1})\right).
\end{align}
Finally, we have
\begin{align*}
& M_k(S(a,\underbrace{b,\dots,b}_{p\ \mathrm{times}},
          \underbrace{b+1,\dots,b+1}_{q\ \mathrm{times}})) \\
& \leq M_k(S(a,\underbrace{b+1,\dots,b+1}_{p+q\ \mathrm{times}}))
     - p\left(M_k(P_{b+1})\!-\!M_k(P_b)\right) 
     & (\mathrm{by }~(\ref{eq-walk-inequality-1})) \\
& = M_k(S(\underbrace{a+1,\dots,a+1}_{p+q\ \mathrm{times}},b))
     + (p\!+\!q\!-\!1)\left(M_k(P_{b+1})\!-\!M_k(P_{a+1})\right) \\
& \quad\ - p\left(M_k(P_{b+1})-M_k(P_b)\right) 
     & (\mathrm{by }~(\ref{eq-moment-canceling})) \\
& = M_k(S(\underbrace{a+1,\dots,a+1}_{p+q\ \mathrm{times}},b))
     + (q\!-\!1)\left(M_k(P_{b+1})\!-\!M_k(P_{a+1})\right) \\
& \quad\ + p\left(M_k(P_{b})-M_k(P_{a+1})\right) \\
& \leq M_k(S(\underbrace{a+1,\dots,a+1}_{p+q\ \mathrm{times}},b+(q-1)(b-a)+p(b-a-1))) 
     & (\mathrm{by }~(\ref{eq-walk-inequality-2})) \\
& = M_k(S(\underbrace{a+1,\dots,a+1}_{p+q\ \mathrm{times}},f)).
\end{align*}
Repeated application of Lemma~\ref{le-path-difference},
that was needed to obtain (\ref{eq-walk-inequality-1}) and~(\ref{eq-walk-inequality-2}),
shows that strict inequality holds above for all sufficiently large values of~$k$,
concluding the proof that
$S(a,\underbrace{b,\dots,b}_{p\ \mathrm{times}},
    \underbrace{b+1,\dots,b+1}_{q\ \mathrm{times}})
\prec S(\underbrace{a+1,\dots,a+1}_{p+q\ \mathrm{times}},f)$.
\end{proof}

\subsection{Closed walks starting at the centers of 
$S(a,b,\dots,b,b+1,\dots,b+1)$ and $S(a+1,\dots,a+1,f)$}

\begin{proposition}
\label{pr-central-closed-walks}
For arbitrary positive integers $a,b$ and $q$ and nonnegative integer~$p$
such that $a<b$ and $p+q\geq 2$,
let $f=(p+q-1)(b-a)+b-p$. Then for all $k\geq 0$
$$
M_k(S(a,\underbrace{b,\dots,b}_{p\ \mathrm{times}},
    \underbrace{b+1,\dots,b+1}_{q\ \mathrm{times}}),u)
\leq M_k(S(\underbrace{a+1,\dots,a+1}_{p+q\ \mathrm{times}},f),u'),
$$
where $u$ and $u'$ are the centers of the respective starlike trees.
\end{proposition}

\begin{proof}
To describe the embedding of closed $k$-walks of 
$S(a,\underbrace{b,\dots,b}_{p\ \mathrm{times}}, 
 \underbrace{b\!+\!1,\dots,b\!+\!1}_{q\ \mathrm{times}})$ starting at~$u$
into closed $k$-walks of 
$S(\underbrace{a+1,\dots,a+1}_{p+q\ \mathrm{times}},f)$ starting at $u'$, 
let us name the appropriate parts of these starlike trees (see Fig.~\ref{fig-D}).
For each $i=1,\dots,p$,
let $v_i$ be the neighbor of~$u$ in the $i$-th branch of length~$b$ 
in $S(a,\underbrace{b,\dots,b}_{p\ \mathrm{times}}, 
 \underbrace{b+1,\dots,b+1}_{q\ \mathrm{times}})$,
and for $i=p+1,\dots,p+q$,
let $v_i$ be the neighbor of~$u$ in the $(i-p)$-th branch of length~$b+1$.
For $i=1,\dots,p+q$,
let $B_i$ denote the branch containing $v_i$, 
but without the vertex~$u$ and the edge~$uv_i$.
Let $D$ denote the remaining branch of length~$a$
in $S(a,\underbrace{b,\dots,b}_{p\ \mathrm{times}}, 
 \underbrace{b+1,\dots,b+1}_{q\ \mathrm{times}})$,
together with the center~$u$.
Similarly,
for $i=1,\dots,p+q$,
let $v'_i$ be the neighbor of~$u'$ in the $i$-th branch of length~$a+1$
in $S(\underbrace{a+1,\dots,a+1}_{p+q\ \mathrm{times}},f)$,
and let $B'_i$ be the branch containing $v'_i$, 
but without the vertex~$u$ and the edge~$u'v'_i$.
Let $F$ denote the remaining branch of length~$f$
in $S(\underbrace{a+1,\dots,a+1}_{p+q\ \mathrm{times}},f)$,
together with the center~$u'$.

\begin{figure}[ht]
\hspace*{36pt}\includegraphics[scale=0.4]{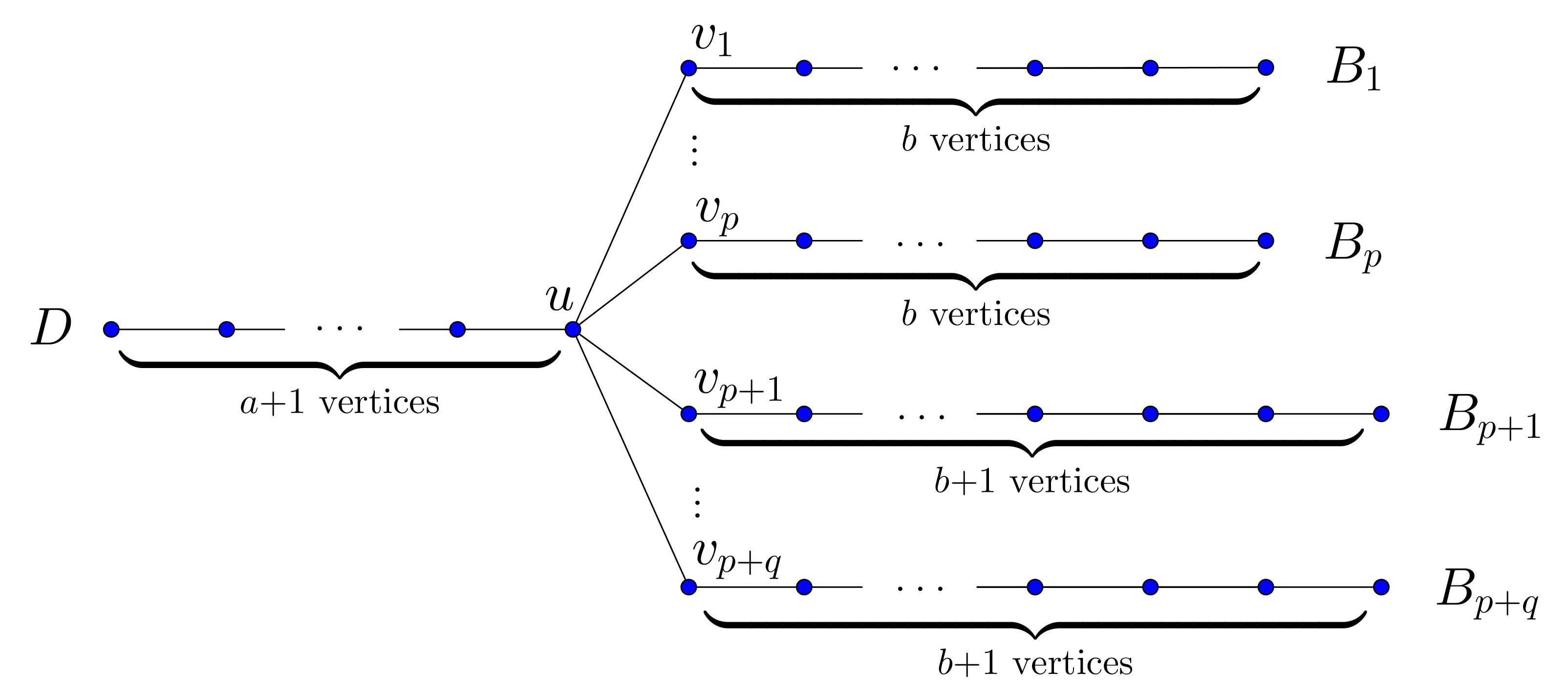} \\[12pt]
\hspace*{10pt}\includegraphics[scale=0.4]{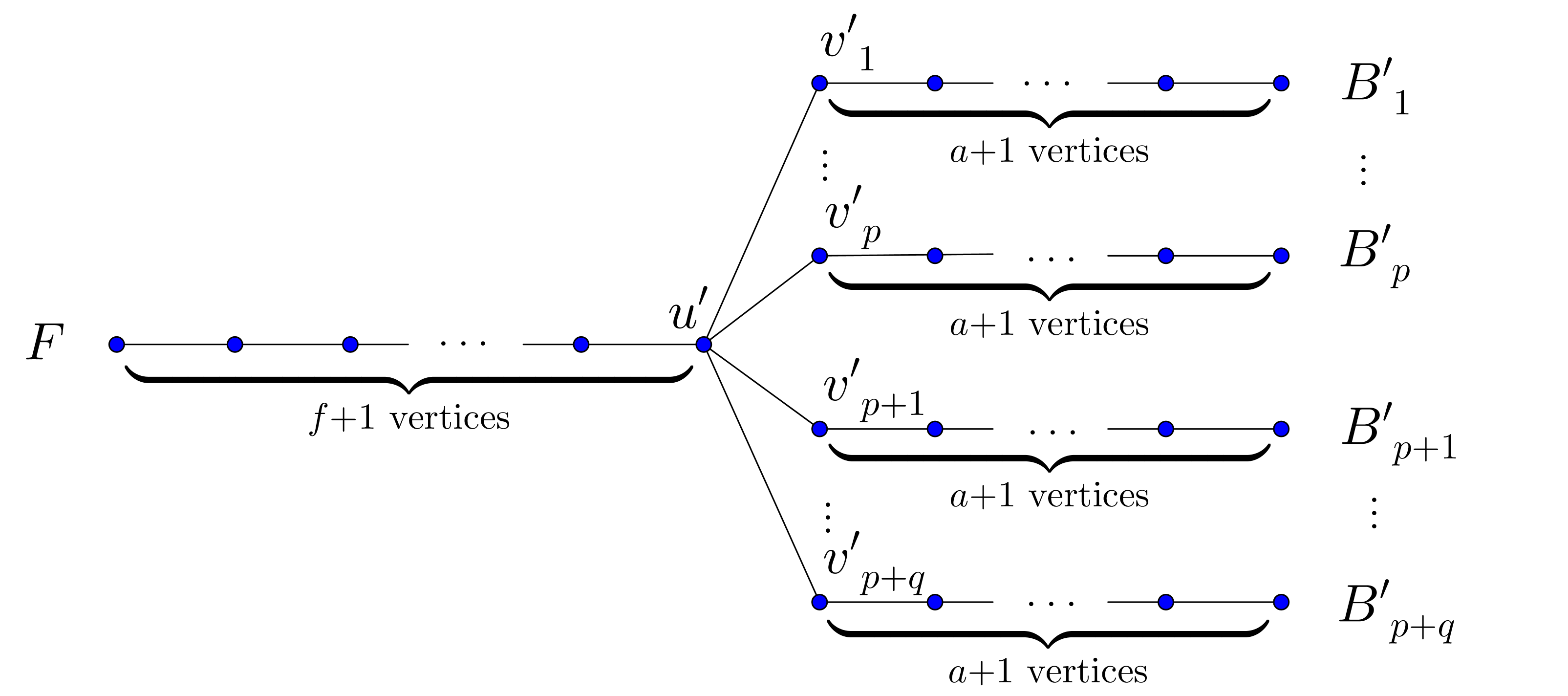}
\caption{Starlike trees $S(a,b,\dots,b,b+1,\dots,b+1)$ and $S(a+1,\dots,a+1,f)$.}
\label{fig-D}
\end{figure}

Further, for $i=1,\dots,p+q$
let $r_i$ be an isometric embedding that maps the branch~$B_i$ to~$F$ 
such that $r_i(v_i)=u'$, and 
let $s_i$ be an isometric embedding that maps the branch~$D$ to $B'_i$
such that $s_i(u)=v'_i$.
Also, let $t$ be an isometric embedding that maps the branch~$D$ to~$F$
such that $t(u)=u'$.
These isometric embeddings exist as
the branch~$B_i$ is isomorphic to either the path $P_{b}$ or $P_{b+1}$
and $F$ is isomorphic to the path $P_{f+1}$, where $f\geq b\geq a+1$,
while $D$ and each $B'_i$ are isomorphic to the path~$P_{a+1}$.

Now, let $W$ be a closed $k$-walk of 
$S(a,\underbrace{b,\dots,b}_{p\ \mathrm{times}}, 
     \underbrace{b+1,\dots,b+1}_{q\ \mathrm{times}})$
that starts at~$u$.
Appearances of edges $uv_1$,\dots,$uv_{p+q}$ in~$W$ 
may be used to represent it in the form
\begin{align*}
W\colon u,Y_0,
      & u,v_{i_1},W_1,v_{i_1},u,Y_1, \\
      & \dots, \\
      & u,v_{i_m},W_m,v_{i_m},u,Y_m,u,
\end{align*}
where the closed walk $u,Y_j,u$ (that possibly consists of~$u$ only)
fully belongs to the branch~$D$ for $j=0,\dots,m$,
while the closed walk $v_{i_j},W_j,v_{i_j}$ (that possibly consists of $v_{i_j}$ only)
fully belongs to the branch~$B_{i_j}$ for $j=1,\dots,m$.
Now set
\begin{align*}
Q(W)\colon u',t(Y_0),
         & u',v'_{i_1},s_{i_1}(Y_1),v'_{i_1},u',r_{i_1}(W_1), \\
         & \dots, \\
         & u',v'_{i_m},s_{i_m}(Y_m),v'_{i_m},u',r_{i_m}(W_m),u'. \\
\end{align*}
$Q(W)$ is a closed $k$-walk in $S(\underbrace{a+1,\dots,a+1}_{p+q\ \mathrm{times}},f)$,
as the first and the last vertex of~$Y_0$ are mapped by~$t$ to a neighbor of~$u'$ in~$F$,
the first and the last vertex of~$Y_j$ for $j=1,\dots,m$
are mapped by~$s_{i_j}$ to a neighbor of $v'_{i_j}$ in $B'_{i_j}$,
while the first and the last vertex of~$W_j$ for $j=1,\dots,m$
are mapped by $r_{i_j}$ to a neighbor of $u'$ in~$F$.
Less formally, 
the initial part of~$W$ until the first edge~$uv_{i_j}$ is 
translated by $Q$ from $D$ to $F$,
while in the remaining parts of~$W$, 
that are delineated by appearances of the edges~$uv_{i_j}$,
the $B_{i_j}$ part is translated to~$F$, the $D$ part is translated to~$B'_{i_j}$ 
and their positions in the walk are switched.
One can then use appearances of the edges $u'v'_{i_j}$ in~$Q(W)$
to determine $t(Y_0)$ and $s_{i_j}(Y_j)$ and $r_{i_j}(W_j)$ for $j=1,\dots,m$
and then use the inverses $t^{-1}$, $s^{-1}_{i_j}$ and $r^{-1}_{i_j}$
to uniquely reconstruct the original walk~$W$,
showing that $Q$ is an injective map.
\end{proof}

\section{Completing the proof of the main theorem}

Propositions \ref{pr-all-closed-walks} and~\ref{pr-central-closed-walks} 
complete the proof of Proposition~\ref{pr-Case-II}.
Propositions \ref{pr-Case-I}, \ref{pr-Case-III} and~\ref{pr-Case-II} now show that 
whenever two partitions $\alpha=(a_1,\dots,a_k)$ and $\beta=(b_1,\dots,b_l)$ of~$n$
are consecutive in the shortlex order with $\alpha<^{\mathrm{lex}}\beta$, then
$$
S(\alpha)\prec S(\beta).
$$
The shortlex order is a linear order on the set of partitions of~$n$,
so that for arbitrary two different partitions $\pi$ and~$\tau$ of~$n$ holds
either $\pi<^{\mathrm{lex}}\tau$ or $\tau<^{\mathrm{lex}}\pi$ and
there exists partitions $\pi_0,\pi_1,\dots,\pi_m$ such that
$\{\pi_0,\pi_m\}=\{\pi,\tau\}$ and 
for each $i=0,\dots,m-1$ 
partitions $\pi_i$ and~$\pi_{i+1}$ are consecutive in the shortlex order
with $\pi_i<^{\mathrm{lex}}\pi_{i+1}$.
Then by Propositions \ref{pr-Case-I}, \ref{pr-Case-III} and~\ref{pr-Case-II}
we have
$$
S(\pi_0)\prec S(\pi_1)\prec\cdots\prec S(\pi_m)
$$
showing that 
$$
\pi<^{\mathrm{lex}}\tau
\quad\Leftrightarrow\quad
S(\pi)\prec S(\tau),
$$
which completes the proof of our main result, Theorem~\ref{th-main}.

At the end,
let us note that the results presented here 
for the ordering of starlike trees by the numbers of closed walks
can be extended analogously to the ordering by the numbers of all walks,
which is still isomorphic to the shortlex ordering of their sorted branch lengths.

\bigskip
\footnotesize
\noindent\textit{Acknowledgments.}
This research was partly supported by research project ON174033 of 
the Ministry of Education, Science and Technological Development of the Republic of Serbia.
The author is also indebted to Stephan Wagner, Francesco Belardo and Milan Pokorny
for reading and discussing the initial draft of the article.

\end{document}